\documentclass[reqno,12pt]{amsart}

\usepackage{color} 
\usepackage{ifpdf}
\ifpdf
    \usepackage[pdftex]{graphicx}
    \usepackage[pdftex]{hyperref}
    \hypersetup{
        unicode=false,          
        pdftoolbar=true,        
        pdfmenubar=true,        
        pdffitwindow=false,     
        pdfstartview={FitH},    
        pdftitle={MCP Article},      
        pdfauthor={Sara Pollock},   
        pdfsubject={Mathematics},    
        pdfcreator={Sara Pollock},  
        pdfproducer={Sara Pollock}, 
        pdfkeywords={PDE, analysis, numerical analysis}, 
        pdfnewwindow=true,      
        colorlinks=true,        
        linkcolor=red,          
        citecolor=blue,         
        filecolor=magenta,      
        urlcolor=cyan           
    }

\else
    \usepackage{graphicx}
    \usepackage{pstricks}
    
    \newcommand{\href}[2]{#2}
\fi

\usepackage{times}
\usepackage{amsfonts}

\usepackage{amsmath}
\usepackage{amsthm}
\usepackage{amssymb}
\usepackage{amsbsy}
\usepackage{amscd}
\usepackage{extarrows}

\usepackage{enumerate}
\usepackage{verbatim}
\usepackage{subfigure}

\newtheorem{theorem}{Theorem}[section]

\newtheorem{assumption}[theorem]{Assumption}
\newtheorem{proposition}[theorem]{Proposition}

\newtheorem{remark}[theorem]{Remark}

\numberwithin{equation}{section}  

  \newcounter{mnote}
  \let\oldmarginpar\marginpar
    \renewcommand\marginpar[1]{\-\oldmarginpar[\raggedleft\footnotesize #1]%
    {\raggedright\footnotesize #1}}

\newenvironment{enumerateX}
{\begin{list}{\arabic{enumi})}{
\usecounter{enumi}
\leftmargin 2.5em\topsep 0.5em\itemsep -0.0em\labelwidth 50.0em}}
{\end{list}}

\definecolor{myblue}{rgb}{0.2,0.2,0.7}
\definecolor{mygreen}{rgb}{0,0.6,0}
\definecolor{mycyan}{rgb}{0,0.6,0.6}
\definecolor{myred}{rgb}{0.9,0.2,0.2}
\definecolor{mymagenta}{rgb}{0.9,0.2,0.9}
\definecolor{mywhite}{rgb}{1.0,1.0,1.0}
\definecolor{myblack}{rgb}{0.0,0.0,0.0}

\newcommand{\beq}{\begin{equation}}
\newcommand{\eeq}{\end{equation}}
\newcommand{\beqa}{\begin{eqnarray}}
\newcommand{\eeqa}{\end{eqnarray}}
\renewcommand{\div}{{\operatorname{div}}}

\newcommand{\eps}{\varepsilon}

\newcommand{\PP}{{\mathbb P}}       
\newcommand{\R}{{\mathbb R}}       

\newcommand{\cL}{{\mathcal L}}

\newcommand{\cT}{{\mathcal T}}

\DeclareMathAlphabet{\mathpzc}{OT1}{pzc}{m}{it}

\newcommand{\bit}{\begin{itemize}}
\newcommand{\eit}{\end{itemize}}

\newcommand{\f}{\frac}

\newcommand{\an}{\text{ and }}

\newcommand{\tforall}{\text{ for all }}

\newcommand{\rest}{\big|}

\newcommand{\range}{\text{range\,}}

\newcommand{\argmin}{\text{argmin\,}}

\newcommand{\grad}{\nabla} 

\newcommand{\goto}{\rightarrow}

\newcommand{\norm}[1]{\ensuremath{\lVert{#1} \rVert}}

\newcommand{\nr}[1]{\norm{#1}} 

\newcommand{\pa}{\partial}

\newenvironment{qqq}{\begin{eqnarray*}\begin{split}\end{split}}{\end{eqnarray*}}
\newcommand\bqq{\begin{qqq}}
\newcommand\eqq{\end{qqq}}

\newenvironment{dsub}[2]{  \begin{array}{ccccccccccccccc}{#1} \\ {#2}}{\end{array} }
\newcommand\bml{\begin{dsub}}
\newcommand\eml{\end{dsub}}

\newenvironment{mat}{\left(\begin{array}{ccccccccccccccc}}{\end{array}\right)}
\newcommand\bcm{\begin{mat}}
\newcommand\ecm{\end{mat}}

\newenvironment{rmat}{\left(\begin{array}{rrrrrrrrrrrrr}}{\end{array}\right)}
\newcommand\brm{\begin{rmat}}
\newcommand\erm{\end{rmat}}

\definecolor{blue}{rgb}{0.2,0.2,0.7}
\definecolor{red}{rgb}{0.7,0.3,0.1}
\definecolor{cyan}{rgb}{0.2,0.5,0.6}

\usepackage{mathtools}
\usepackage{stmaryrd}

\include{sp_sym}

\begin{document}

\title[A regularized Newton-like Method]
      {A regularized Newton-like method for nonlinear PDE}

\author[S. Pollock]{Sara Pollock}
\email{snpolloc@math.tamu.edu}

\address{Department of Mathematics\\
         Texas A\&M University\\ 
         College Station, TX 77843}


\date{\today}


\begin{abstract}
An adaptive regularization strategy for stabilizing Newton-like iterations on a coarse mesh is developed in the context of adaptive finite element methods for nonlinear PDE.
Existence, uniqueness and approximation properties are known for finite element solutions of quasilinear problems assuming the initial mesh is fine enough.  
Here, an adaptive method is started on a coarse mesh where the finite element discretization and quadrature error produce a sequence of approximate problems with indefinite and ill-conditioned Jacobians.  
The methods of Tikhonov regularization and pseudo-transient continuation are related and used
to define a regularized iteration using a positive semidefinite penalty term.  
The regularization matrix is adapted with the mesh refinements and its scaling is adapted with the iterations to find an approximate sequence of coarse mesh solutions leading to an efficient approximation of the PDE solution.  Local  q-linear convergence is shown for the error and the residual in the asymptotic regime and numerical examples of a model problem illustrate distinct phases of the solution process and support the convergence theory.
\end{abstract}

\maketitle




\section{Introduction}
This note discusses a practical regularization technique used in an adaptive finite element method (AFEM) for nonlinear and in particular stationary quasilinear problems.  
The method uses Tikhonov regularization to stabilize a sequence of coarse-mesh problems leading to an adaptively generated mesh where the approximate problem is well-posed.  It is demonstrated using a positive semidefinite penalty matrix based on the Laplacian, essentially penalizing iterates for being far from $H^2$.  
This stabilization technique is particularly relevant in the context of adaptive methods where quadrature approximations of the problem data on a coarse mesh may produce linearizations with sharp spikes and ill-conditioned and indefinite Jacobians.
The goal is to stabilize the coarse mesh problem enabling local refinement to enrich the solution space first to resolve the problem data, and then to obtain an accurate solution.  The alternative of starting on a uniformly fine mesh creates an unnecessarily large problem as many degrees of freedom are added where the data is smooth and already well approximated; moreover, it may not be \emph{a priori} known how fine of an initial mesh is necessary run the computation without stabilization.

Borrowing from the techniques used in  ill-posed problems,  Tikhonov regularization is applied as in for example~\cite{Engl1996,EKN89,SEK93,Kalten97,Jin08,Bonesky09,HoMa12} and the references therein, to  stabilize the possibly indefinite Jacobian with an appropriate penalty term in an automated way, ultimately attaining an efficient adaptive method once the problem data is resolved.  The resulting Newton-like iteration on each mesh refinement is similar to a well analyzed and motivated method found in the work of Bank and Rose in~\cite{BaRo80a} and \cite{BaRo81a} for uniformly monotone problems; and more generally as the pseudo-transient continuation ($\Psi$tc) method discussed in the context of Newton-methods in~\cite{Deuflhard11} and developed in for  example~\cite{KeLi13,FoKe05,CoKeKe02,KeKe98}. The discussion here relates the Tikhonov-regularization point of view to that of pseudo-transient continuation and extends the $\Psi$tc theory to penalty matrices with adaptively expanded nullspaces.

This work differs from the well-developed adaptive framework for solving nonlinear PDE in~\cite{ErnVor13} as this presentation addresses the coarse mesh regime dominated by quadrature error where the
nonlinear solver is designed to approach a stable state, even if to a local rather than global minimizer in the interest of determining the next mesh refinement.  In the context of adaptive finite difference methods for two-point boundary value problems, \cite{SmMa85} discusses the critical mesh-spacing necessary for the Newton iterations to remain in the domain of convergence across mesh refinements.  The presence of an analogous critical mesh-spacing is detected in the numerical examples presented here, and the analysis of this is under investigation by the author.

The remainder of the paper is organized as follows. Section~\ref{sec:TRsetup} describes the quasilinear problem and its finite element implementation.  Section~\ref{sec:regular} introduces the Newton-like iterations and relates the concepts of Tikhonov regularization and pseudo-transient continuation.  Section~\ref{sec:phases} describes three phases of the solution process and provides local convergence results for the error and residual in the final, asymptotic phase.  Section~\ref{sec:numerics} demonstrates the method and the distinct solution phases on a model problem.

\section{Problem setup}\label{sec:TRsetup}
The goal of an adaptive method is the efficient approximation to the solution of a problem, achieved by selectively increasing the degrees of freedom.  The focus of this discussion is developing a stabilized solver for an adaptive method where due to sharp gradients and near-singularities in the data, the nonlinear problem is not well resolved on one or more of the initial coarse meshes and an approximate coarse mesh solution must be found in order to run the adaptive method.

Throughout this discussion, $(u(x) , v(x) ) = \int_\Omega u(x) v(x)\, dx$, and similarly for vector-valued functions.  
Consider the quasilinear PDE in divergence form
\begin{align}\label{eqn:qLstrong}
F(u) \coloneqq -\div( \kappa(u) \grad u) - f(x,y) = 0 \text{ in } \Omega \subset \R^2, \quad u = 0 \text{ on } \pa \Omega,
\end{align}
for polygonal domain $\Omega$ and $F: X \goto Y^\ast$ with $F'(u) \in \cL(X, Y^\ast)$ for real Banach spaces $X \an Y$. 
For $f \in L_2(\Omega) \cap L_\infty(\Omega)$ and $\kappa(u)$ bounded away from zero with  $l$th derivative  $D^{(l)}(\kappa(s))$ bounded for $l = 0, 1, 2$ as in~\cite{CaRa94}, then there is a unique solution $u \in W^{1,p}(\Omega)$, with $2 < p < \infty$. This problem also  fits into the context of~\cite{Xu96} with the assumption that $\kappa(u)$ is bounded and $F'(u): H_0^1(\Omega) \goto H^{-1}(\Omega)$ is an isomorphism, in which case $u$ is an isolated solution. 
The weak form of~\eqref{eqn:qLstrong} is given by, find $u \in X$ such that
\begin{align}\label{eqn:qLweak}
B(u,v) \coloneqq ( \kappa(u) \grad u, \grad v ) - (f,v)= 0 \tforall v \in Y.
\end{align}
The discretized equation is, find $u_h \in X_h$ such that 
\begin{align}\label{eqn:qLdiscrete}
B(u_h,v)=0 \tforall v \in Y_h,
\end{align}
 where $X_h \subset X$ and $Y_h \subset Y$ are discrete finite element spaces with respect to triangulation $\cT_h$, where the family of triangulations $\{\cT_h\}_{0 < h < 1}$ is regular and quasi-uniform  in the sense of~\cite{Ciarlet}.  Computationally, rather than solving~\eqref{eqn:qLdiscrete}, one solves an approximate discrete problem as given by, find $u_h \in X_h$ such that
\begin{align}\label{eqn:qLapprox}
B_h(u_h, v) \coloneqq ( \kappa_h(u_h) \grad u_h, \grad v ) - (f_h,v)=0 \tforall v \in Y_h,
\end{align}
where $\kappa_h \an f_h$ are approximations to $\kappa \an f$ evaluated by quadrature on triangulation $\cT_h$. On a sufficiently fine mesh,~\eqref{eqn:qLapprox} is a good approximation on~\eqref{eqn:qLdiscrete}, and quadrature error may be neglected in the analysis.  
Following~\cite{CaRa94} and the references therein, under the above assumptions  the inf-sup conditions hold for problem~\eqref{eqn:qLstrong} and the discrete inf-sup conditions hold for problem~\eqref{eqn:qLdiscrete}, assuming the meshsize $h$ is small enough; see also ~\cite{Xu96} and~\cite{Schatz74} for an alternative discussion describing the small meshsize condition.  In particular, taking $X_h = Y_h = V_h $ the finite element space of linear Lagrange finite elements  $\PP_1$ over $\cT_h$ that satisfy the homogeneous Dirichlet boundary conditions,
\begin{align}\label{eqn:distreteinfsup}
\inf_{\substack{v_h \in V_h \\ |v_h|_{1,p}= 1}} \sup_{\substack{w_h \in V_h \\ |w_h|_{1,q}= 1}}
( F'(u)v_h, w_h ) \ge \beta_0 > 0
\end{align}
with $X = W_0^{1,p}(\Omega), ~Y = W_0^{1,q}(\Omega) \an F'(u):W_0^{1,p}(\Omega) \goto W^{-1,p} = Y^\ast$, for $p > 2$ and   $1/p + 1/q = 1$.  This yields existence, uniqueness and approximation properties of $u_h$ the solution to the exact discrete problem~\eqref{eqn:qLdiscrete}.
The discussion now focuses on the pre-asymptotic regime in which ~\eqref{eqn:qLapprox} is not a good approximation to ~\eqref{eqn:qLdiscrete}, the meshsize is not sufficiently small  and the discrete inf-sup condition~\eqref{eqn:distreteinfsup} is not assumed. However, the sequence of coarse-grid problems are still of use in obtaining a mesh where quadrature error becomes negligible and~\eqref{eqn:distreteinfsup} holds.
The remainder of the paper describes an asymptotically efficient algorithm to stabilize and solve the nonlinear problem~\eqref{eqn:qLapprox} with a regularized Newton-like method. 

\section{Regularized methods}\label{sec:regular}
On each refinement of the mesh, the solution to the nonlinear problem~\eqref{eqn:qLapprox} is approximated iteratively. The following notation is used in the remained of the paper: the $n$th iteration subordinate to partition $\cT_h$ is denoted $u^n_h$, or $u^n_k$ on the $k$th partition $\cT_{k}$; $u^n$ is the $n$th iteration on a fixed partition and $u_k$ is the final iteration on the $k$th mesh, taken as the approximate solution on $\cT_k$.
For the standard Newton method starting with initial guess $u^0$, iterate until convergence:
\begin{align}\label{alg:Newton}
\text{solve: } & A^n x^n = F^n \nonumber \\
\text{update: } & u^{n+1} = u^n + x^n,
\end{align}
 where $A^n$ is the Jacobian matrix associated with~\eqref{eqn:qLapprox}
  with respect to the finite element basis, found by taking the Gateaux derivative of $F(v)$ in direction $w$  given by 
$
F'(v)w = \f{d}{dt}F(v + tw)\rest_{t = 0}.
$ 
With  $V_h = X_h = Y_h$ and $\{ \phi_i\}$ the set of basis functions for $V_h$,  the Jacobian matrix  for iteration $n$ on partition $\cT_h$ is assembled by
\begin{align}\label{eqn:jacobian}
A_{ij}^n = (\kappa_h'(u^n_h) \phi_i \grad u^n_h, \grad \phi_j) + (\kappa_h(u^n_h) \grad \phi_i, \grad \phi_j),
\end{align} 
where all functions of $u^n_h$ in~\eqref{eqn:jacobian} are evaluated by quadrature.
The residual vector is assembled by
\begin{align}\label{eqn:residual}
F_{j}^n =  - \{( \kappa_h(u_h^n) \grad u_h^n, \grad \phi_j ) - (f_h,\phi_j)\}.
\end{align} 

Newton's method, algorithm~\eqref{alg:Newton} converges quadratically for $u^0$ close enough to $u^\ast$ the solution to $F(u^\ast)= 0$, and otherwise may converge slowly or fail to converge altogether.  
Following the ideas used in the solution of ill-posed problems as for example in~\cite{Engl1996,Bonesky09,HoMa12}, consider stabilizing the real $d$-dimensional matrix problem at each iteration by Tikhonov regularization.   The solution $x^n$ of   $A^n x^n = F^n$ for invertible matrix $A^n$ is the unique solution of the minimization problem
\begin{align}
x^n = \argmin_{x \in \R^d} \nr{F^n - A^n x}^2.
\end{align}
More generally, $w_\alpha^n$ is the unique minimizer of the Tikhonov functional 
\begin{align}\label{eqn:TikMin}
w_\alpha^n = \argmin_{w \in \R^d}G_\alpha(w), \quad ~\text{ with}~
 G_\alpha(w) \coloneqq \{ \nr{F^n - A^n w}^2 + \alpha \nr{Rw}^2 \},
\end{align}
for stabilization matrix $R$ and parameter $\alpha$.  The necessary and sufficient optimality condition on $w_\alpha^n$ requires $G_\alpha'(w_\alpha^n)$ = 0, resulting in the regularized normal equations
\begin{align}\label{eqn:NormalIt}
 \left(\alpha_n R^TR + A^{n^T}A^n \right)w_\alpha ^n = A^{n^T}F^n.
\end{align}
The solution $w_\alpha^n$ to~\eqref{eqn:NormalIt} is the Tikhonov approximation to $x^n$, the solution to \eqref{alg:Newton}. As $A^n$ and $F^n$ are constructed by~\eqref{eqn:jacobian} and~\eqref{eqn:residual} with respect to the approximate problem~\eqref{eqn:qLapprox} rather than~\eqref{eqn:qLdiscrete}, both can be viewed as noisy data.
As the noise due to quadrature error pollutes both $F^n$ and $A^n$, 
the problem departs from the analysis of linear ill-posed problems where only the right-hand side data is considered noisy and an \emph{a priori} bound of $\nr{F^n  - F^n_\delta} \le \delta$, with noise level characterized by $\delta$ is used in the convergence analysis.
Here, the $\alpha_n$ and $R_h$ should be chosen so that $\alpha_n R_h^TR_h\goto 0$ as both $\kappa_h \goto \kappa$ and $f_h \goto f$.  

\begin{remark}
Also unlike the case of ill-posed problems, the equation $A^n u^n = F^n$  assumes $F^n \in \range(A^n)$. The normal equations are used to create stability when the Jacobian matrix $A^n$ is indefinite, ensuring the shift in the spectrum of the left-hand side operator is away rather than towards zero.
For a coercive elliptic problem with variable diffusion coefficients, the discrete approximation to the bilinear form  even with quadrature error can also be shown coercive~\cite{BS08}.  Tikhonov regularization of a linear problem in that case can be observed to stabilize the initial sequence of coarse-mesh solutions, but an eventual solution on an adaptive mesh can be achieved with or without regularization.
\end{remark}

For the problem considered here, let $R$ the Laplacian stiffness matrix assembled by
\begin{align}\label{eqn:Lpenalty}
R_{ij} = (\grad \phi_i, \grad \phi_j).
\end{align}
There is significant literature on the choice of regularization parameters for ill-posed problems, see for instance~\cite{EKN89,SEK93,Engl1996,Kalten97,Jin08,Bonesky09,HoMa12} and the references therein; and an analysis of parameter choice for this discrete well-posed if noisy problem is the topic of current investigation by the author. In the scope of this article, examples of regularization parameters $\alpha_n \an R_h$ are demonstrated to effectively stabilize the model problem and allow accurate approximation of~\eqref{eqn:qLweak}.  
With $A^n$ and $F^n$ assembled by~\eqref{eqn:jacobian} and~\eqref{eqn:residual} obtain the regularized Newton-like iteration
\begin{align}\label{alg:NTRNewton}
\text{solve: } & \left(\alpha_n R^TR + A^{n^T}A^n \right)w_\alpha ^n = A^{n^T}F^n \nonumber \\
\text{update: } & u^{n+1} = u^n + w_\alpha^n.
\end{align}
Tikhonov regularization theory is based on the normal equations formulation~\eqref{alg:NTRNewton}. In the special case where the Jacobians $A^n$ are known to be positive definite, the following secondary formulation which exploits the sparsity of the Jacobian may be used.  
\begin{align}\label{alg:TRNewton}
\text{solve: } & (\alpha_n R + A^n )w_\alpha^n = F^n \nonumber \\
\text{update: } & u^{n+1} = u^n + w_\alpha^n.
\end{align}
While more efficient due to maintaining the sparse structure of the original problem, method~\eqref{alg:TRNewton} may actually shift a nonpositive spectrum close to zero and is generally the less robust of the two methods. 
Formulation~\eqref{alg:NTRNewton} creates stability for indefinite Jacobians but with a loss of sparse structure: it stabilizes the computation on the coarsest grids but it is preferable to switch to the algorithm~\eqref{alg:TRNewton} once the discretization yields positive Jacobians.  Determining efficiently computable criteria for choosing which formulation to use is a topic of further investigation by the author, and is presently  performed based on the size of the initial residual on each mesh refinement. If stability is the priority, the adaptive algorithm can be run with~\eqref{alg:NTRNewton} alone.

\begin{remark}
A note is made in~\cite{BaRo80a} regarding the relation of the ``s" method and a method of the form~\eqref{alg:TRNewton} if the penalty term $R$ is symmetric positive definite.  If $R$ is invertible, then \eqref{alg:TRNewton} fits into the framework of pseudo-transient continuation as described in for example~\cite{Deuflhard11,CoKeKe02,FoKe05,KeKe98}; and the ``s" method of Bank and Rose falls into this category.  
For the solution $u^*$ of the nonlinear problem $g(u)=0$, the iteration~\eqref{alg:TRNewton} is given by
$(\alpha_n R + g'(x^n))w^n = g(x^n)$, where $R$ is a positive semidefinite linear functional with adjoint $R
^\ast$. The ``s" method is given by
\begin{align}\label{alg:smethod}
\text{solve: }  \left( \f{1}{s_n} I + g'(x^n) \right)w^n = g(x^n), \quad u^{n+1} = u^n + w^n,
\end{align}
and can be thought of as a homotopy or continuation method to determine a path from initial guess $u^0$ to steady state $u^\ast$  by integrating the ODE
\begin{align}\label{mot:ode_s}
\f{du}{dt} + g(u(t)) = 0, \quad u(0) = u^0.
\end{align}
Setting $s_n = t_{n+1} - t_{n}$, and letting $u^n$ approximate $u(t_n)$, ~\eqref{mot:ode_s} is discretized by
\begin{align}\label{mot:ode_Ds}
\f{d}{dt}u^{n+1}  + g(u^{n+1}) = 0.
\end{align}
Applying a backward Euler approximation to $d (u^{n+1})/dt$ and linearizing $g(u^{n+1})$ about $g(u^n)$, obtain the iteration~\eqref{alg:smethod}.   The present method~\eqref{alg:TRNewton} may be similarly described by considering the ODE
\begin{align}\label{mot:ode_R}
\f{d}{dt} (Ru(t)) + g(u(t)) = 0, \quad u(0) = u^0,
\end{align}
in place of~\eqref{mot:ode_s} for linear operator $R$. In this sense, $R$ may be though of as regularizing the path from $u^0$ to $u^\ast$ with $\alpha_n = 1/s_n$ large enough, \emph{i.e., } the step-size $s_n$ chosen small enough for stability.  Notably, there is no requirement for $R$ to be invertible and in some cases it is preferable to choose $R$ so the exact solution $u^\ast$ lies in or at least close to the nullspace of $R$. Method~\eqref{alg:NTRNewton} based on the normal equations fits into the $\Psi$tc context by applying the same discretization to the ODE
\begin{align}\label{mot:ode_RTR}
\f{d}{dt} (R^\ast Ru(t)) + g'(u(t))^\ast g(u(t)) = 0, \quad u(0) = u^0
\end{align}
with $g'(u^{n+1})^\ast$ the formal adjoint of $g'(u^{n+1})$ approximated by $g'(u^n)^\ast$.
In light of this observation, both~\eqref{alg:NTRNewton}  and~\eqref{alg:TRNewton} can be seen as Tikhonov regularization with the strength of the penalty term controlled by $\alpha$ or as pseudo-transient continuation with timestep $s = 1/\alpha$. 
\end{remark}

Ideally, the exact solution $u$ should be in the nullspace of the regularization matrix; another view is the regularization should penalize against undesirable properties of the iterates such as sharp spikes or high curvature determining a smoothed path for the homotopy method to follow.  Choice of an appropriate matrix then depends on \emph{a priori} knowledge of the solution, for instance its regularity.  The Laplacian for example penalizes against spikes in curvature.  Development of more precise and localized penalty terms is under investigation by the author; the examples in Section~\ref{sec:numerics} use a sequence of singular penalty matrices found by applying a cutoff to the Laplacian penalty matrix,  regularizing only degrees of freedom where the error spikes as determined by an \emph{a posteriori} error indicator.  This method of adaptively expanding the nullspace of the penalty matrix with each mesh refinement has been observed to reduce the time spent in Newton-like iterations by approximately 85\% as compared to the full-rank regularization technique when the algorithm is run from a coarse mesh until asymptotic error reduction is achieved.

\section{Solution Process}\label{sec:phases}
In the spirit of~\cite{CoKeKe02,KeKe98} the adaptive method is discussed in three  phases. Similarly to~\cite{SmMa85} and the recent~\cite{ErnVor13}, iterations are are stopped and the mesh is refined some number of times during each phase, characterized as follows.

\textit{Initial phase: } Quadrature error dominates and the iterations defined by~\eqref{alg:NTRNewton} use a large penalty term or small timestep.  Iterations are stopped when decrease of the residual slows and the mesh is refined using an \emph{a posteriori} error indicator.  If the residual fails to decrease, the coarsest elements of the mesh are refined and the solution process is restarted with the initial guess for the next refinement reset to zero. At present, there are no convergence results for this phase; at worst, the coarse mesh elements are refined sufficiently-many times to bring the algorithm to the next phase.

\textit{Pre-asymptotic phase: } Quadrature error is still large but iterates defined by~\eqref{alg:NTRNewton} converge to a smooth local minimizer of the Tikhonov functional $G_\alpha(w)$. On each mesh refinement $\cT_k$ the penalty term starts large and decreases guided by the norm of the residual, \textit{i.e.,} $\alpha_n = \gamma_n \nr{F^n}$.  Iterations are stopped when decrease in residual slows or the residual drops below a preset tolerance;  the mesh is refined using an \emph{a posteriori} error indicator, and the Newton-like iterations on each new refinement are started with the  interpolation of the solution from the previous refinement onto the current mesh.

\textit{Asymptotic error reduction: } The problem is well approximated and the mesh may be assumed sufficiently fine so existence and uniqueness results for the nonlinear problem and its finite element approximation apply. In this phase the method may be switched from~\eqref{alg:NTRNewton} to~\eqref{alg:TRNewton} as the problem is bigger in terms of total dof and it is advantageous to exploit the sparsity of the Jacobian.  The regularization term is small and may be phased out altogether. Iterations are stopped when  the residual decreases below a preset tolerance;  the mesh is refined using an \emph{a posteriori} error indicator, and the Newton-like iterations on each new refinement are started with the interpolation of the solution from the previous refinement onto the current mesh.

In the initial phase, iterates at best approach local minimizers.  The mesh is coarse and the problem is relatively small but the overall efficiency of the algorithm could be improved by determining early stopping criteria when iterates fail to converge. This is a topic of current investigation by the author.

\subsection{Local Convergence}\label{subsec:localconvergence}
Local convergence describes the behavior of the method in the asymptotic error reduction phase, where the initial guess on a given refinement is sufficiently close to a minimizer.  The relevance of this analysis is the regularization term does not prevent the iterates from converging to the correct solution.  An analysis addressing the pre-asymptotic regime will be discussed in future work, for instance by considering the sequence of approximate problems as snapshots of a stiff PDE. 

For clarity of presentation, the algorithm is analyzed for finding a zero of function $g(x)$.  Relating the notation to the assembled matrices~\eqref{eqn:jacobian} and~\eqref{eqn:residual} used in computation, $A^n$ discretizes $g'(u^n)$ and $F^n$ discretizes $g(u^n)$.
Let $x^\ast$ the solution to $g(x) = 0$ and denote the open ball $B(x^\ast, \eps) = \{ x \, \rest \, \nr{x - x^\ast} < \eps\}$.  Denote the error $e^n = x^n - x^\ast$.

\begin{assumption}\label{assume1} (\textit{c.f.} Assumptions 2.2-2.3 of~\cite{CoKeKe02}).  There exist $\beta, \omega_L \an \eps_1 > 0$ so that for positive semidefinite $R$, and for all 
$0 < \alpha_n < \alpha_M$, then for all $x \in B(x^\ast, \eps_1)$:
\begin{enumerateX}
\item $\nr{g'(x) - g'(y) } \le \omega_L \nr{x - y}$ for all $y \in B(x^\ast, \eps_1)$.
\item $ \alpha_n R + g'(x)$ is invertible.
\item $\nr{(\alpha_n R + g'(x))^{-1}} \le M_I$.
\item $ \nr{(\alpha_n R + g'(x))^{-1} (\alpha_n R)} \le \f{\alpha_n/\beta}{1 + \alpha_n/\beta}$.
\end{enumerateX}
\end{assumption}

Assumption~\ref{assume1} (4) characterizes the stability of the approximate Jacobian, \textit{c.f.,} Assumption 2.3 in~\cite{CoKeKe02} and Assumption 2.1.3 in~\cite{KeKe98}.  Unlike the standard Newton method, it is not assumed here that $g'(x)$ is invertible or has a bounded inverse, and the role of the regularization term $\alpha_n R$ is to bound the inverse of the approximate Jacobian away from zero in a well-conditioned way in the sense of Assumption~\ref{assume1} (4).  With these assumptions which describe the asymptotic phase of the algorithm, local convergence of the error is shown, and under similar assumptions local convergence of the residual is  then demonstrated.  It is relevant to show this for algorithm~\eqref{alg:TRNewton} as the problem data has stabilized enough to take advantage of the sparse approximate Jacobian in this phase.
 
 \begin{theorem}\label{theorem:localerror}
 Let $\alpha_n < \alpha_M$ and let  Assumption~\ref{assume1} hold.  Then there is $\eps_E$ so that for $x^n \in B(x^\ast, \eps_E)$, the iteration~\eqref{iter:TRNewton}  converges $q$-linearly. The regularized iteration is given by
 \begin{align}\label{iter:TRNewton} 
x^{n+1} - x^n = - (\alpha_n R + g'(x^n))^{-1} g(x^n).
 \end{align}
 \end{theorem}
 
 \begin{proof}
 \begin{align}\label{eqn:EC001}
 e^{n+1} &= e^n - (\alpha_n R + g'(x^n))^{-1} g(x^n) \nonumber \\
 & = (\alpha_n R + g'(x^n))^{-1}( g'(x^n)e^n - g(x^n)) +
 (\alpha_n R + g'(x^n))^{-1}( \alpha_n R) e^n.
 \end{align}
 
 The first term in~\eqref{eqn:EC001} is bounded by Assumption~\ref{assume1} and the integral mean value theorem together with the Lipschitz assumption on $g'(x)$.
 \begin{align}
 \nr{g'(x^n) e^n - g(x^n)} & = \nr{ g'(x^n) e^n - \int_0^1 g'(x^\ast + t e^n) e^n \, dt} \nonumber \\
 & \le \omega_L \int_0^1 \nr{e^n}^2 (1-t) \, dt  = \f {\omega_L}2 \nr{e^n}^2,
 \end{align}
 yielding
 \begin{align}\label{eqn:EC002}
\nr{  (\alpha_n R + g'(x^n))^{-1}( g'(x^n)e^n - g(x^n)) } \le M_I \f{\omega_L}{2} \nr{e^n}^2.
 \end{align}
 The second term in~\eqref{eqn:EC001} is bounded by Assumption~\ref{assume1} (4).  Then 
 \begin{align}\label{eqn:EC003}
 \nr{e^{n+1}} & \le \nr{e^n} \left( M_I \f{\omega_L}{2} \nr{e^n} + \f{1}{1+\beta/\alpha_n} \right).
 \end{align}
 To establish $q$-linear convergence, set 
 \begin{align}\label{eqn:EC003b}
 \eps_2 = \f{1 }{(M_I  \omega_L) } \f{\alpha_M/\beta}{(1+ \alpha_M/\beta)(1/2 +  \alpha_M/\beta)},
 \quad \an \eps_E = \min\{\eps_1, \eps_2\}.
 \end{align}
 Then, $M_I \omega_L \nr{e^n}/2 + (1 + \beta/\alpha_n)^{-1} < ( 1 + \beta/(2 \alpha_n))^{-1}$ for 
 $x^n \in B(x^\ast, \eps)$, and from~\eqref{eqn:EC003}
 \begin{align}\label{eqn:EC004}
 \nr{x^{n+1} - x^\ast} \le \nr{x^n - x^\ast}  \left( \f{\alpha_n/\beta}{ 1/2 + \alpha_n/\beta}\right).
 \end{align}
 Rearranging~\eqref{eqn:EC003} yields
 \[
  \nr{e^{n+1}}  \le C_A  \nr{e^n}
  \left(  \nr{e^n} +  \left( \f{\alpha_n}{\beta} \right)   \left( \f{1}{ C_A} \right) \right), \quad \text{with}~ C_A =  \f{M_I  \omega_L}{2} ,
 \]
yielding asymptotically quadratic convergence assuming the sequence $\alpha_n \goto 0$.
 \end{proof}
 
 \begin{remark}\label{rem:basinsize}
 Maximizing~\eqref{eqn:EC003b} for $\eps = \eps(y(\alpha))$ with $y = \alpha/\beta$, obtain
 \[
 \max_{y > 0} \eps(y) = (6 - 4 \sqrt 2) \f{1}{M_I \omega_L} \approx  \f{1}{3M_I \omega_L},
 \]
 with the maximum attained at $\alpha = \beta/ \sqrt 2$.  The constant $\beta$ is \textit{a priori} unavailable computationally, however this shows making $\alpha$ arbitrarily large does not enlarge the domain of convergence of the algorithm.  While it is not practical to test at every refinement, this also yields the condition based on Assumption~\ref{assume1} (4), suggesting
  $ \nr{(\alpha_n R + g'(x))^{-1} (\alpha_n R)} \le \f{1}{1 + \sqrt 2}$ as a guideline for q-linear convergence of method~\eqref{alg:TRNewton}.
 \end{remark}
 
A similar analysis shows the local $q$-linear and asymptotically quadratic convergence of the residual.  The following bound similar to Assumption~\ref{assume1} (4) is required.
 
 \begin{assumption}\label{assume2} There exist $ \beta_S \an \eps_3 > 0$ with $\eps_3 \le \eps_E$ so that for positive semidefinite $R$ and for all 
$0 < \alpha_n <\alpha_M$, then for all $x \in B(x^\ast, \eps_3)$:
\[
 \nr{ (\alpha_n R)(\alpha_n R + g'(x))^{-1}} \le \f{\alpha_n/\beta_S}{1 + \alpha_n/\beta_S}.
 \]
\end{assumption}

\begin{theorem}\label{theorem:localresi}
Let the hypotheses of Theorem~\ref{theorem:localerror} and Assumption~\ref{assume2} hold.  
 Then there is $\eps_S>0$ so that for $x^n \in B(x^\ast, \eps_S)$, the sequence of residuals defined by iteration~\eqref{iter:TRNewton}  converges $q$-linearly to zero. 
 \end{theorem}

 \begin{proof}
Let $\Delta x^n = x^{n+1} - x^n$.  
By the integral mean value theorem and iteration~\eqref{iter:TRNewton}
\begin{align}\label{eqn:RC001}
g(x^{n+1}) &= g(x^n) + \int_0^1 g'(x^n + t \Delta x^n) \Delta x^n \, dt  \nonumber \\
& = g(x^n)+ g'(x^n) \Delta x^n + \int_0^1\left[  g'(x^n + t \Delta x^n) - g'(x^n)\right] \Delta x^n\, dt
\nonumber  \\
& = \alpha_n R(\alpha_nR + g'(x^n))^{-1} g(x^n)
+ \int_0^1\left[  g'(x^n + t \Delta x^n) - g'(x^n)\right] \Delta x^n\, dt. 
\end{align}
 The first term of~\eqref{eqn:RC001} is bounded by Assumption~\ref{assume2} and the second using the Lipschitz condition, Assumption~\ref{assume1} (1). Then
 \begin{align}\label{eqn:RC002}
 \nr{g(x^{n+1}) } & \le \f{1}{1 + \beta_S/\alpha_n} \nr{g(x^n)} + \f{\omega_L}{2} \nr{\Delta x^n}^2
 \nonumber \\
 & \le  \nr{g(x^n)} \left( \f{1}{1 + \beta_S/\alpha_n} + \f{\omega_L M_I}{2} \nr{\Delta x^n} \right).
 \end{align}
 By Theorem~\ref{theorem:localerror} for $x^n \in B(x^\ast,\eps_E)$ the next iterate $x^{n+1} \in B(x^\ast,\eps_E)$ so that $\nr{\Delta x^n}  \le 2 \eps_E$.  
 To establish $q$-linear convergence of the residual, choose $\eps_4$ small enough so that by~\eqref{eqn:RC002} and the same reasoning as in Theorem~\ref{theorem:localerror}
 \begin{align}\label{eqn:RC003}
 \nr{g(x^{n+1})} \le \nr{ g(x^n)}  \left( \f{\alpha_n/\beta_S}{ 1/2 + \alpha_n/\beta_S}\right).
 \end{align}
From~\eqref{eqn:RC002} and iteration~\eqref{iter:TRNewton}
 \begin{align}\label{eqn:RC004}
  \nr{g(x^{n+1}) } & \le  \nr{g(x^n)} \left( \f{1}{1 + \beta_S/\alpha_n} + \f{\omega_L M_I^2}{2} \nr{g(x^n)} \right)
  \nonumber \\
  & \le  C_B \nr{g(x^n)} \left(  \nr{g(x^n)} +  \left( \f{\alpha_n}{\beta_S} \right)  \f{1}{C_B}  \right), \quad
  \text{with}~ C_B = \f{\omega_L M_I^2}{2}, 
 \end{align}
 yielding asymptotically quadratic convergence to zero of the residual assuming the sequence $\alpha_n \goto 0$. 
 \end{proof}
 
\begin{proposition} With the domain of convergence as given by Theorem~\ref{theorem:localerror} and Theorem~\ref{theorem:localresi}, the sequences given by
\[
\alpha_n = \nr{g(x^n)}, ~\an~ \alpha_n = \f{\nr{g(x^n)}^2}{\nr{g(x^{n-1})}},
\]
both lead to asymptotically quadratic convergence of the residual, as predicted by~\eqref{eqn:RC004}.
\end{proposition}
The choice of parameter $\alpha_n = \gamma_n\nr{g(x^n)}$ is justified in~\cite{BaRo80a} and related choices are  discussed in~\cite{Deuflhard11,KeKe98} and the references therein.
The parameter choice used in the present results $\alpha_n = {\nr{g(x^n)}^2}/{\nr{g(x^{n-1})}}$ and the relevant convergence theory for method~\eqref{alg:NTRNewton} will be further discussed in future investigations by the author.

\section{Numerical Examples}\label{sec:numerics}
Consider the quasilinear stationary diffusion equation
\[
F(u) \coloneqq-\div( \kappa(u) \grad u) - f =0 ~\text{ in }~ \Omega = [0,1] \times [0,1], \an u = 0 \text{ on } \pa \Omega,
\]
with nonlinear diffusion coefficient
\[
\kappa(s) = 1 + \f 1{\left( (\epsilon + (s - a)^2\right)},  \quad ~\text{with}~ \eps = 10^{-3}, \quad a = 0.5,
\]
and load $f(x,y)$  chosen so the exact solution $u(x,y) = \sin \pi x \sin \pi y$.
Existence and uniqueness of solutions is discussed in Section~\ref{sec:TRsetup} following~\cite{CaRa94} and~\cite{Xu96}, assuming the mesh is fine enough.

 The problem is discretized with finite element space $V_k$ consisting of linear Lagrange finite elements  $\PP_1$ over partition $\cT_k$ that satisfy the homogeneous Dirichlet boundary conditions.
On an initial mesh of  288 elements or less the Jacobian based on certain iterates $u^n_k$ is observed to be indefinite, justifying the use of~\eqref{alg:NTRNewton}.

The algorithm is implemented using the finite element library FETK~\cite{Holst2001a} and a direct solver is used on each linear system.
The mesh is refined with respect to local element indicators.  In the numerical examples that follow, standard residual-based indicators are used both for mesh refinement and to determine which degrees of freedom to stabilize.  

The local \emph{a posteriori} residual-based indicator for element $T \in \cT_k$ with $h_T$ the element diameter is given by
\[
\eta_T^2(v) = \eta_{\cT_k}^2(v,T) \coloneqq h_T^2 \nr{ F(v)}_{L_2(T)}^2 +  h_T \nr{ J_T(v)  }_{L_2(\pa T)}^2,
\]
$J_T(v) \coloneqq \llbracket [\kappa(v) \grad v  \cdot n \rrbracket_{\pa T}$,
with jump
$\llbracket \phi \rrbracket_{\pa T} \coloneqq {\lim_{t \goto 0} \phi(x + t n) - \phi(x - tn)}$, 
 where  $n$ is the appropriate outward normal defined on $\pa T$.

On each mesh partition $\cT_k$ the penalty matrix $R = R_{k}$ is adapted from the Laplacian stiffness matrix as in~\eqref{eqn:Lpenalty},  modified by a cutoff against degrees of freedom  $v_j$ for which the local error indicator $\eta_T = \eta_T(u_k^0)$ satisfies
\[
\eta_T \le \sqrt{\text{median}_{T \in \cT_k}(\eta_T)},
\]
for each element $T$ which contains $v_j$ as a vertex. An investigation of conditions for selective application of  regularization is currently being investigated by the author.  This heuristic cutoff is used because of the observed separation in the magnitude of the local error indicators.  It is also observed with this cutoff that none of the degrees of freedom get selected for additional regularization within a few mesh refinements past  the point where the error decays at its asymptotic rate.

The regularization parameter $\alpha_n$ that scales the penalty matrix $R_k$ is chosen as follows: Set $\gamma_0 = 1$. For $n \ge 1$,
\begin{align}\label{eqn:alpha_prac}
\alpha_{n} = \gamma_n \nr{F^{n}}, \quad ~\text{with}~ \gamma_{n}  = \f {\nr{F^{n}}}{\nr{F^{n-1}}}.
\end{align}
To reduce rapid fluctuation of $\gamma_n$, correct to ensure $\gamma_{n-1}/2 \le \gamma_n \le 1$ in the case that   $\nr{F^{n}}< \nr{F^{n-1}}$ and $\gamma_n \le 2 \gamma_{n-1}$ if $\nr{F^{n}}> \nr{F^{n-1}}$.

On the first coarse mesh the Newton-like method is started with initial guess $u_0^0 = 0$.  On subsequent refinements, $u_k^0$ is chosen by interpolating $u_{k-1}$ onto $V_k$.  
If the criteria for exiting the Newton-like iterations are not met on partition $\cT_k$ after a maximum allowed number of iterations, the coarsest elements are refined,  $ u^0_{k+1}$ is reset to zero, and the cutoff on the regularization is not used on $R_{k+1}$.  

The criterium used  for switching from method~\eqref{alg:NTRNewton} to method~\eqref{alg:TRNewton} is the initial residual $\nr{F^0} < 50$, tested on each mesh refinement.  
The stopping criteria for the iterations are either
\begin{align}
\nr{F^{n+1}}& \le \text{tol}, \quad \text { or } \label{cond:stop1} \\
\nr{F^{n+1}} &< \nr{F^0} \an \nr{F^{n+1}} < \nr{F^{n}} \an \gamma_{n}> \gamma_{n-1} \label{cond:stop2}.
\end{align}
Condition~\eqref{cond:stop2} effectively stops the iterations when the decrease in the residual slows and allows the mesh to refine.  The tolerance for~\eqref{cond:stop1} in this example is set at tol=$10^{-7}$.

Data is shown for the adaptive algorithm starting from an initial mesh of 36 elements.  Snapshots of the solution and the adaptive mesh are shown for adaptive levels 10, 16 and 20, each representing a phase of the solution process with respect to Section~\ref{sec:phases}.  Tables~\eqref{tab:iter10N}-\eqref{tab:iter20S}  summarize data for the iterations on the same refinements. The data shown are the norm of the residual $\nr{F^{n+1}}$ assembled by~\eqref{eqn:residual} with $u^{n+1}_k$, the scaled difference in iterates $E_n = \nr {(u^{n+1}_k - u^{n}_{k})}/\nr{u^n_k}$, and the factors $\gamma_n$ as determined by~\eqref{eqn:alpha_prac}, and $J_n =  \nr{(\alpha_n R_k + A^n)^{-1} (\alpha_n R_k)}$ from Assumption~\ref{assume1} (4) and Remark~\ref{rem:basinsize}, with $A^n$ assembled by~\eqref{eqn:jacobian} with $u_k^n$. Theory suggests~\eqref{alg:TRNewton} should converge $q$-linearly for $J_n \le 1/(1 + \sqrt 2)$. The iterations after 10 refinements with 320 elements and 16 refinements with 728 elements are run with~\eqref{alg:NTRNewton}.The iterations on the 20th adaptive mesh with 1500 elements are run with formulation~\eqref{alg:TRNewton}.

\begin{table}
	\centering
	\small
\begin{tabular}{c||c|c|c|c}
ITER & $\nr{F^{n+1}}$  & $E_n $ & $\gamma_n$  & $ J_n$\\
\hline
0& 	 95.8162& 			 &	                   &		\\
1& 	 30.5175 &	 0.0336786& 	  1              	&1.16517	\\
2& 	 27.3112 &	 0.0693449& 	 0.5        	&1.1353	\\
3& 	 26.1526 &	 0.0546612& 	 0.894936 &1.14712	 \\
\end{tabular}
\caption{Newton-like iterations using~\eqref{alg:NTRNewton} on a mesh of 320 elements with regularization applied to  136 of 185 dof. }
\label{tab:iter10N}
\end{table}

\begin{table}
	\centering
	\small
\begin{tabular}{c||c|c|c|c}
ITER & $\nr{F^{n+1}}$  & $E_n $ & $\gamma_n$  & $ J_n$\\
\hline
0& 	 101.284& 			&					&					\\
1& 	 254.537& 	 0.167595& 	  	 1			&5.61919		\\
2& 	 64.6942& 	 0.0321575& 	  	 2			&2.09372		\\
3& 	 30.8987& 	 0.0344072& 	  	 1			&1.90928		\\
4& 	 9.56493& 	 0.0393813& 	  	 0.5			&1.35121		\\
5& 	 3.25445& 	 0.0699149& 	  	 0.309558		&0.831689	\\
6& 	 1.89826& 	 0.0635441& 	  	 0.340248		&0.525463	\\
\end{tabular}
\caption{Newton-like iterations using~\eqref{alg:NTRNewton} on a mesh of 728 elements with regularization applied to 309 of 389 dof. }
\label{tab:iter16N}
\end{table}

\begin{table}
	\centering
	\small
\begin{tabular}{c||c|c|c|c}
ITER & $\nr{F^{n+1}}$  & $E_n $ & $\gamma_n$  & $ J_n$\\
\hline
0& 	 17.7935   & 				& 				&				\\
1& 	 1.71252    &	 0.00742287 	 & 	 1			&1.72215			\\
2& 	 0.518451 & 	 0.0218088 	 & 	 0.5			&0.42968			\\
3& 	 0.0738831& 	 0.0095354 	 & 	 0.302742		&0.119834		\\
4& 	 0.00111901& 	 0.0012099 	 & 	 0.151371		&0.0100722		\\
5& 	 1.67738e-07& 	 1.42661e-05 	 & 	 0.0756854	&7.76983e-05		\\
6& 	 2.04483e-12 &	 1.71896e-09 	 & 	 0.0378427	&5.82459e-09		\\
\end{tabular}
\caption{Newton-like iterations using~\eqref{alg:TRNewton} on a mesh of 1500 elements with regularization applied to 642 of 775 dof. }
\label{tab:iter20S}
\end{table}

Tables~\eqref{tab:iter10N} and~\eqref{tab:iter16N} both illustrate the stopping criterium based on the increase in factors $\gamma_n$ effectively stops the iterations when the decrease in the residual slows.  As the theory predicts an increase in the convergence rate as $u^n \goto u$, the the iterates in this case are approaching a stable configuration other than the solution.  Table~\eqref{tab:iter20S} shows data based on the first iteration using the standard method~\eqref{alg:TRNewton} and illustrates the $q$-linear convergence once $J_n < 1/(1 + \sqrt 2)$.  In practice, the factors $J_n$ are not computed, although they could be periodically monitored. 

Table~\eqref{tab:adaptivesum} summarizes the data for the adaptive algorithm starting on the mesh of 36 elements.  The solutions is reset and coarse elements refined if the iterations fail to meet the stopping criteria after 20 iterations. The norm of the final residual $\nr{F^{k}}$ assembled by~\eqref{eqn:residual} from $u_{k}$ and the final regularization factor $\alpha_k$ on each refinement $\cT_k$ are shown, as well as the ratio of regularized to total degrees of freedom. Up to level 19 the iterations are run using~\eqref{alg:NTRNewton} and the rest using~\eqref{alg:TRNewton}.  On the $25$th refinement no dof are selected for refinement and that remains the case for the following iterations, meaning the standard Newton iteration~\eqref{alg:Newton} is run from that point forward. Both $\alpha_k$ and $\nr{F^k}$ show the three distinct solution phases. The solution is only reset during the initial phase and converges on every level in the asymptotic phase.  In between in the pre-asymptotic phase the iterates approach a stable state other than the solution.  

It is observed that starting on an initial mesh of 36, 72, 144 or 288 elements all adaptive refinements reset at the same stages so the mesh and solutions agree in all cases.  Starting the algorithm with an initial mesh of 18 elements the mesh is asymptotically similar in that the coarsest elements are the same size as the other cases when the algorithm runs long enough, however there are minor differences in the element subdivisions on the smaller elements.  This predicts a minimum meshsize necessary to resolve the problem data, \textit{c.f.,}~\cite{SmMa85} for the case of finite differences, and will be investigated in this context by the author in future work. 
\begin{table}
	\centering
	\small
\begin{tabular}{c||c|c|c|c}
Level & iterations & $\nr{F^k}$  & $\alpha_k $ & Reg. dof \\ 
\hline
1 & 7		&1276.48		& 1272.67		&25/30			\\
2 & 6		&1288.69		&1191.57		& 27/31			\\
3 & 20	&1157.12		&1120.33		&30/32			\\
4 &3		&848.942		&846.74		&49/49			\\
5 &20	&858.131		&775.344		&43/55 			\\
6 &10	&482.029		&374.066		&85/85 			\\
7 &20	&756.814		&961.174		& 68/89 			\\
\hline 
8 &12	&38.7518		&23.4659		&169/169			\\
9 &4		&32.3166		&35.1739		&144/177			\\
10 &3	&26.1526		&15.2588		&136/185 			\\
11 &4	&65.5013		&48.1103		&160/205 			\\
12 &3	&63.1959		&49.4333		&176/221			\\
13 &3	&48.9646		&27.8695		&217/266 			\\
14 &4	&30.1089		&15.0512		&247/306 			\\
15 &4	&11.7586		&7.34768		&291/346			\\
16 &6	&1.89826		&1.10732		& 309/389			\\
17 & 4	&1.62103		&0.720371	&361/424			\\
18 &6	&0.000984649	&0.0250257	&439/506 			\\
\hline
19 &7	& 6.4825e-11	& 8.55164e-07	&554/632			\\
20 &6	&2.04483e-12	&6.34765e-09	&642/775 			\\
21 & 4	&2.54008e-10 	&7.64046e-06	&922/1080 		\\
22 & 4	& 2.02901e-10	&2.99153e-06	&1152/1473 		\\
23 &4	& 4.67019e-09	&1.57458e-05	&1253/2097		\\
24 & 4	&5.45452e-11	&1.60574e-06	&362/3061 		\\
25 & 3	&5.75245e-12	&8.32757e-08	& 0/4250			
\end{tabular}
\caption{Summary of data showing the adaptive algorithm through three solution phases.}
\label{tab:adaptivesum}
\end{table}

Figure~\eqref{fig:sol} shows the finite element solution after 10, 16 and 20 adaptive refinements starting with an initial mesh of 36 elements. These snapshots illustrate the three phases of the solution process discussed in Section~\ref{sec:phases} where the initial phase is characterized by nonsmooth solutions; however, the adaptive mesh as shown in Figure~\eqref{fig:mesh} still enriches the degrees of freedom where the data has the sharpest gradients.  The
pre-asymptotic phase is characterized by smooth solutions which fail to converge to the solution of the exact problem; in the asymptotic phase the finite element solutions converge to the exact sinusoidal solution.
\begin{figure}
\includegraphics[trim=0pt 0pt 0pt 20pt, clip=true,width=0.3\textwidth]{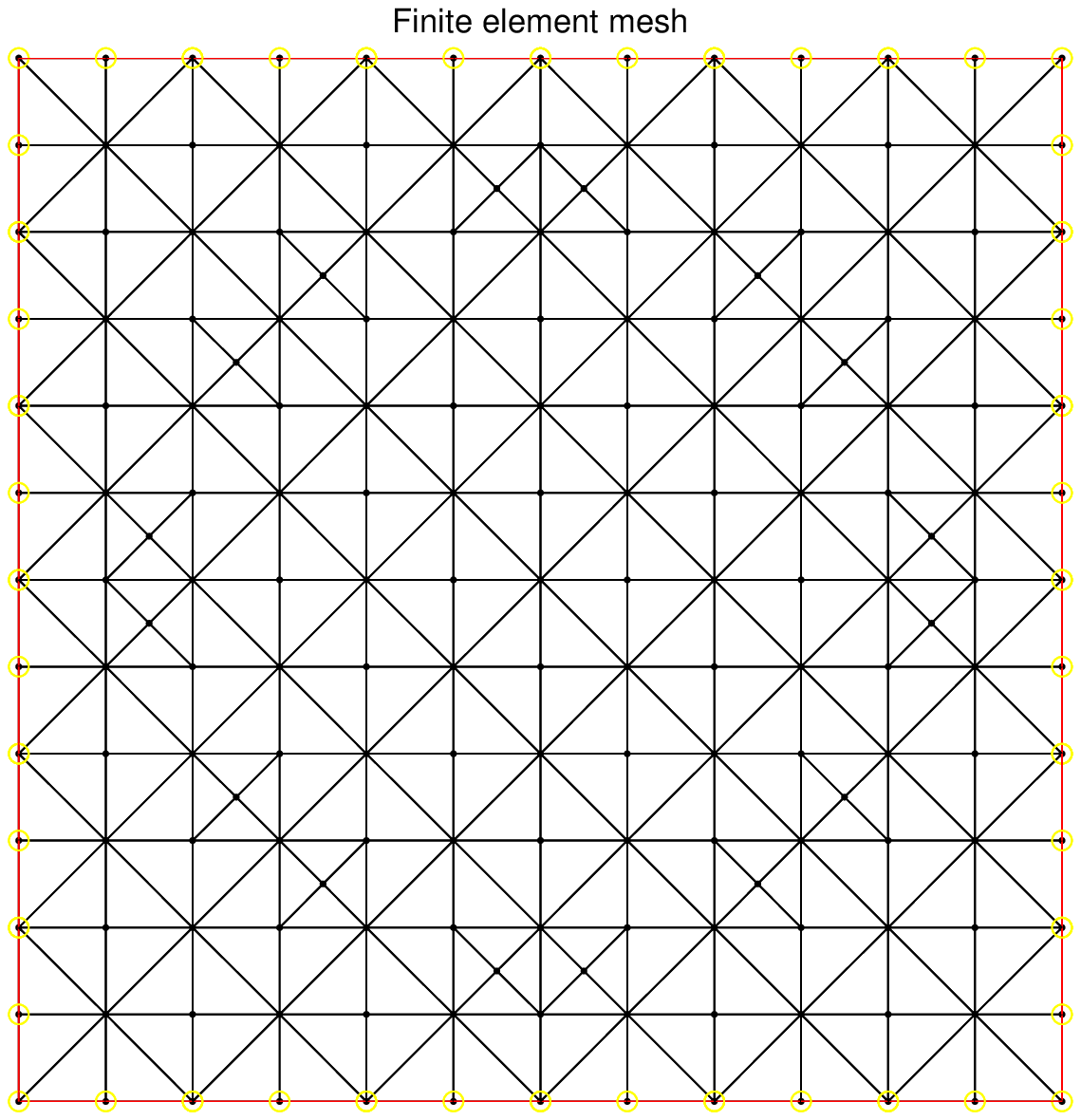}~\hfil~
\includegraphics[trim=0pt 0pt 0pt 20pt, clip=true,width=0.3\textwidth]{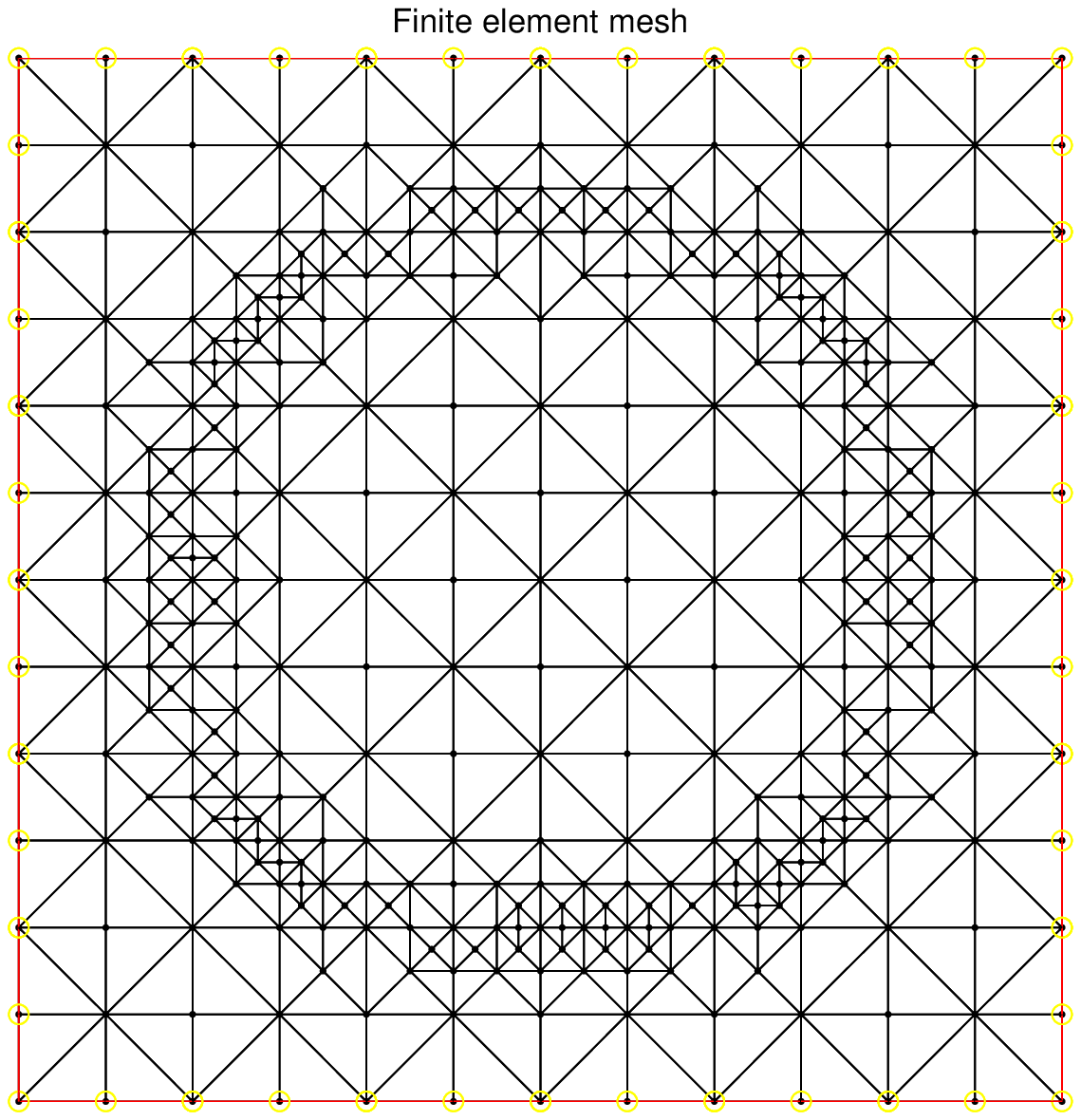}~\hfil~
\includegraphics[trim=0pt 0pt 0pt 20pt, clip=true,width=0.3\textwidth]{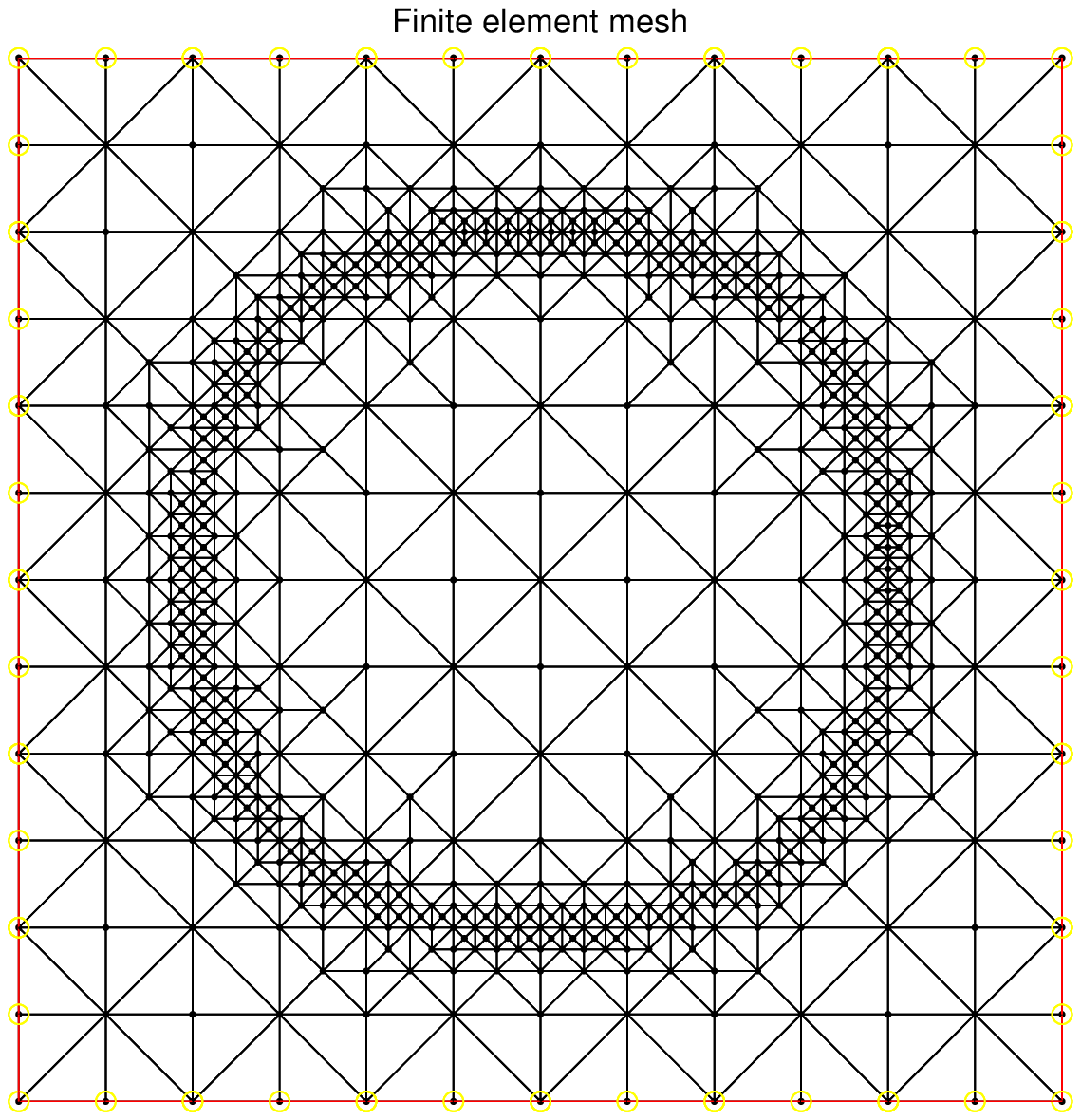}~
\caption{Left: mesh after 10, 16 and 20 adaptive refinements from an initial mesh with 32 elements.}
\label{fig:mesh}
\end{figure}

\begin{figure}
\includegraphics[trim=0pt 0pt 0pt 20pt, clip=true,height=0.35\textwidth]{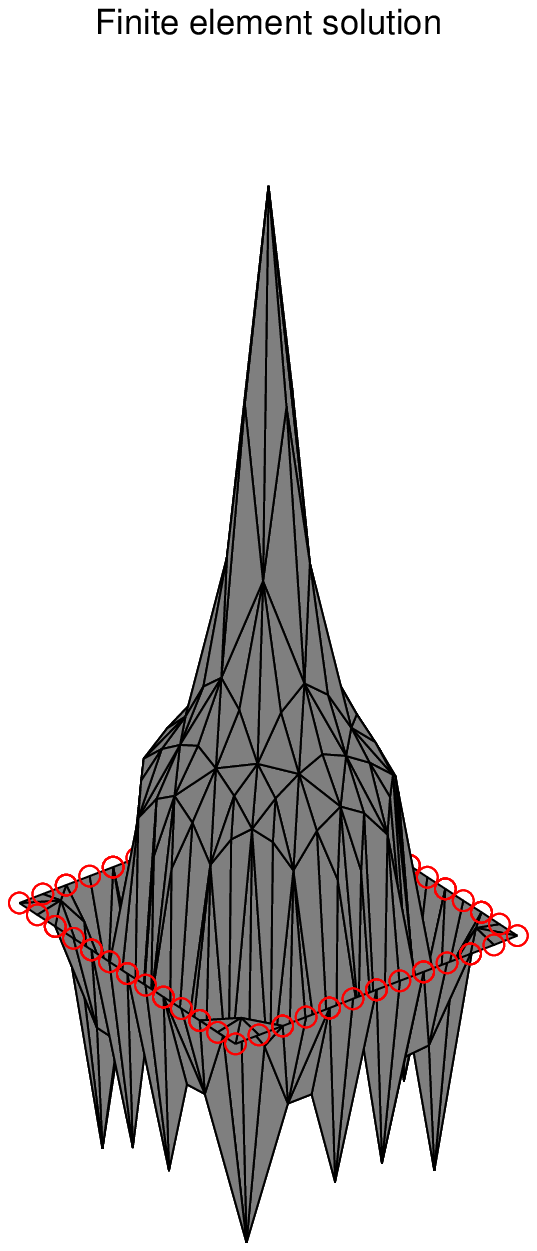}~\hfil~
\includegraphics[trim=0pt 0pt 0pt 20pt, clip=true,height=0.3\textwidth]{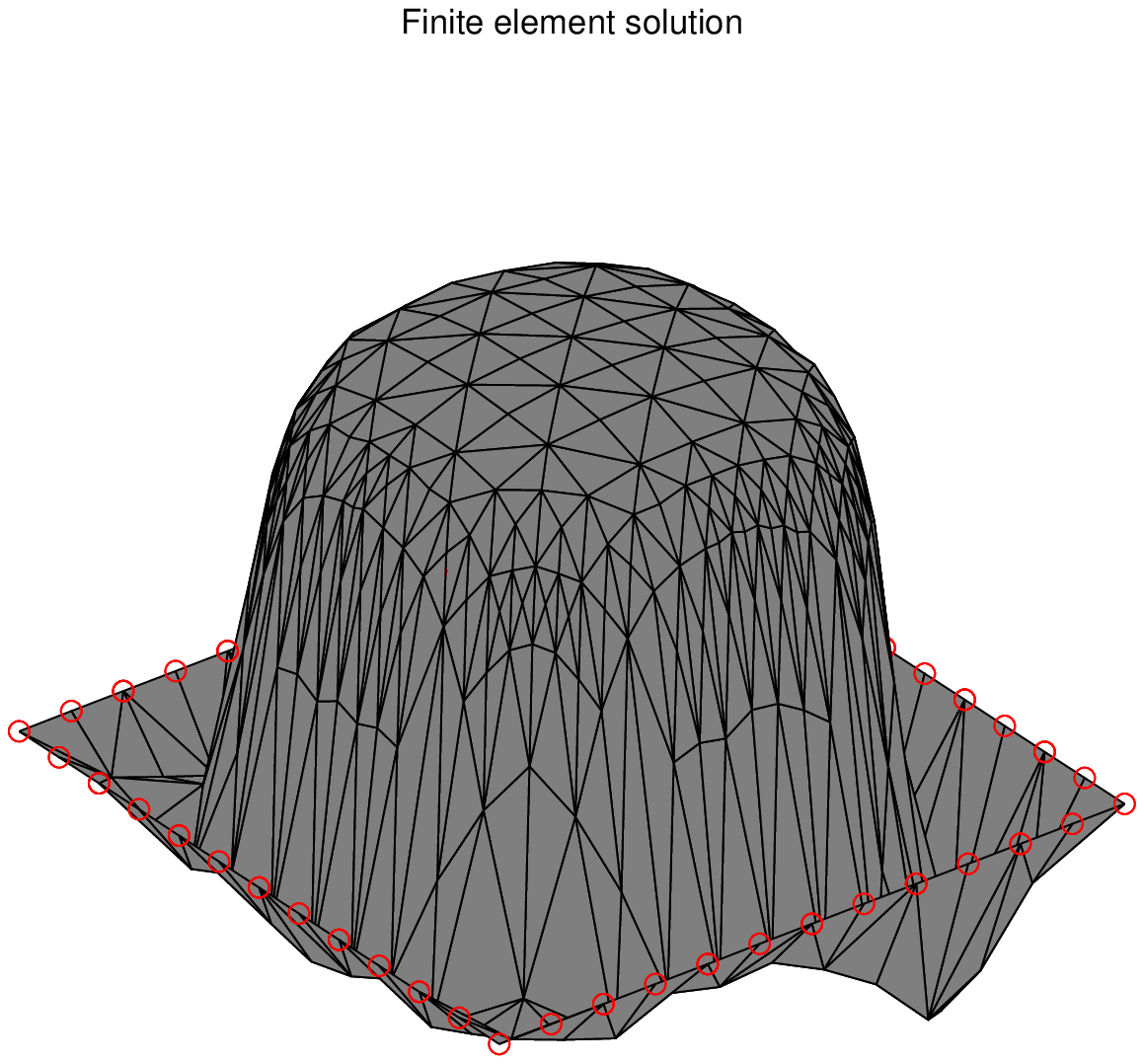}~\hfil~
\includegraphics[trim=0pt 0pt 0pt 20pt, clip=true,height=0.3\textwidth]{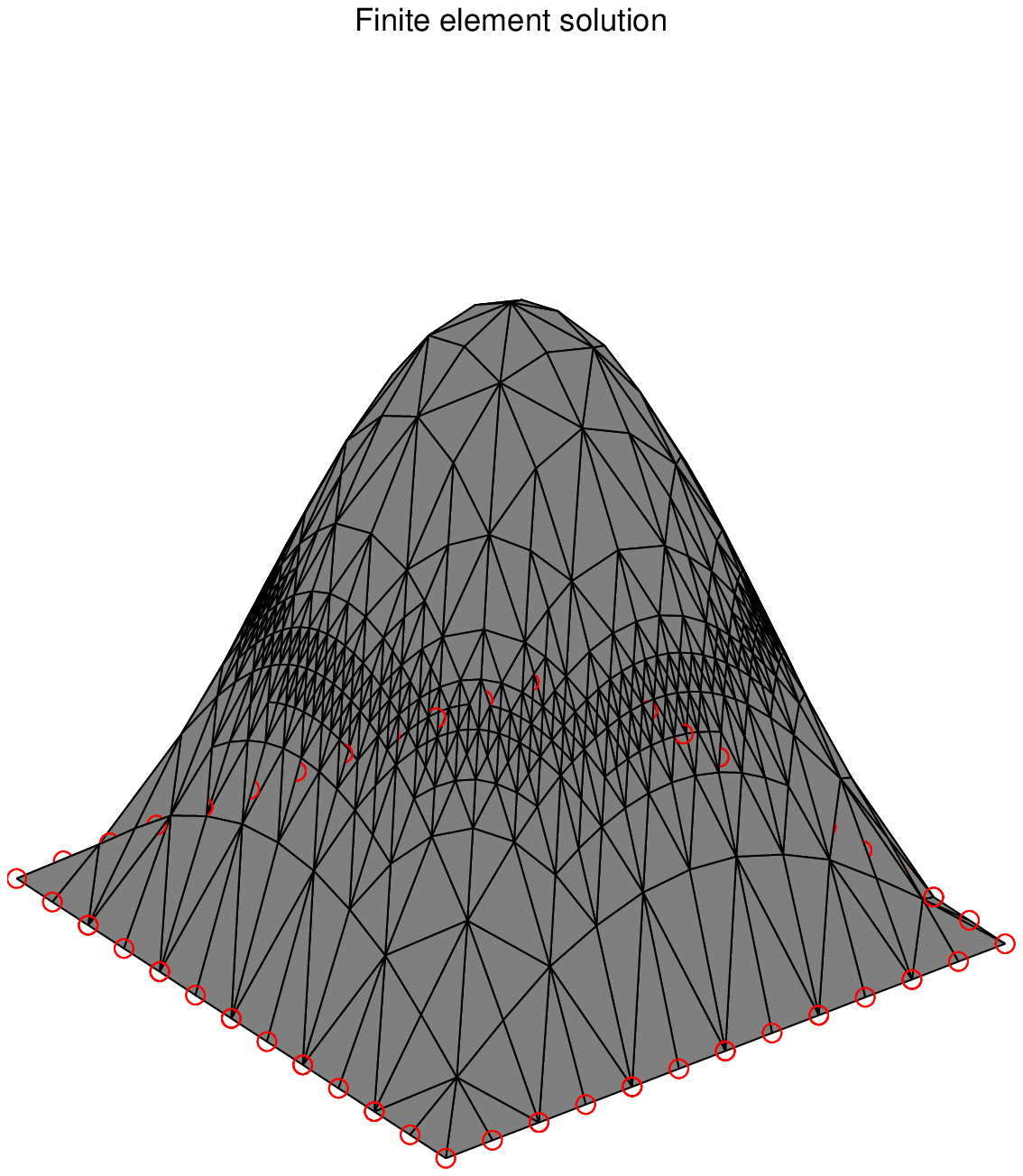}~
\caption{Left: solution after 10, 16 and 20 adaptive refinements from an initial mesh with 32 elements.}
\label{fig:sol}
\end{figure}

Figure~\eqref{fig:h1error}  shows the $H^1$ error (below) and error estimator (above) corresponding to the iterations in Table~\eqref{tab:adaptivesum} starting after the last reset and continuing to the 28th adaptive refinement with 24040 elements. The three solution phases are again clear from  the $H^1$ error and the distinction between the last two phases is clear in the estimator. Once the problem data is resolved, the error and estimator decrease asymptotically like $n^{-1/2}$, with $n$ the number of mesh elements.
\begin{figure}
\includegraphics[width=0.5\textwidth]{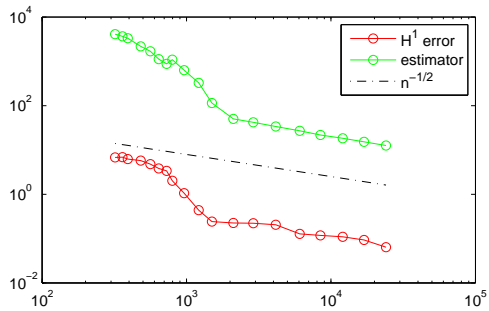}
\caption{$H^1$ error and error estimator corresponding to Table~\eqref{tab:adaptivesum} compared with $n^{-1/2}$  with $n$ the number of mesh elements.  
}
\label{fig:h1error}
\end{figure}

\section{Conclusion}\label{sec:conclusion}
This discussion presents a regularization  strategy  for Newton-like iterations designed to run an adaptive algorithm for a quasilinear problem starting on a coarse mesh. 
Tikhonov regularization used in the solution of ill-posed problems with noisy data is related to pseudo-transient continuation methods used in the computation of steady-state solutions to nonlinear differential equations. The resulting method is demonstrated to produce a convergent sequence of approximations on a model problem where an ill-conditioned indefinite Jacobian makes standard or damped Newton iterations infeasible.  
 The algorithm is discussed with respect to, and numerically demonstrated to display three distinct phases in the solution process, and $q$-linear convergence of the  regularized iterative method is demonstrated in the final asymptotic phase.  Future investigations by the author will address problem-dependent meshsize requirements necessary to transition from the initial to the pre-asymptotic phases, where the existence of such conditions are suggested by the numerics.  Analysis and control of the regularization parameters in the pre-asymptotic phase, as well as criteria for using the regularized formulation based on the normal equations for problems of this form and the investigation of stronger nonlinearities  will also be addressed in future work. 
\section{Acknowledgments}
   \label{sec:ack}
The author would like to thank William Rundell for many conversations
providing insight into this work.

\bibliographystyle{abbrv}
\bibliography{refsTRN}



\end{document}